\documentclass[10pt]{amsart}
\usepackage{amscd,amsmath,amssymb,amsthm,enumerate,graphicx}
\usepackage[bookmarks=false,dvipdfm]{hyperref}
\newtheorem{definition}[equation]{Definition}
\newtheorem{lemma}[equation]{Lemma}
\newtheorem{proposition}[equation]{Proposition}
\newtheorem{theorem}[equation]{Theorem}
\newtheorem{corollary}[equation]{Corollary}
\newtheorem{remark}[equation]{Remark}
\newtheorem{conjecture}[equation]{Conjecture}

\newcommand\lemmaref[1]{Lemma~\ref{#1}}
\newcommand\propositionref[1]{Proposition~\ref{#1}}
\newcommand\theoremref[1]{Theorem~\ref{#1}}
\newcommand\corollaryref[1]{Corollary~\ref{#1}}

\newcommand\conjectureref[1]{Conjecture~\ref{#1}}

\title{Kobayashi's conjecture on associated varieties for $(\mathrm{E}_{6(-14)},\mathrm{Spin}(8,1))$}
\author{Haian HE}
\date{}
\address{Department of Mathematics, College of Sciences, Shanghai University, No. 99 Shangda Road, Baoshan District, Shanghai, 200444 China P. R.}
\email{hebe.hsinchu@yahoo.com.tw}
\begin{document}
\begin{abstract}
The author confirms a conjecture on associated varieties by Toshiyuki KOBAYASHI for the Klein four symmetric pair $(\mathrm{E}_{6(-14)},\mathrm{Spin}(8,1))$, which provides an alternative way to confirm the conjecture for the symmetric pair $(\mathrm{Spin}(8,2),\mathrm{Spin}(8,1))$. Also, for Klein four symmetric pairs $(G,G^\Gamma)$ with the exceptional simple Lie groups $G$ of Hermitian type, there exists a discrete series representation of $G$ which is $G^\Gamma$-admissible if and only if $(G,G^\Gamma)$ is of holomorphic type.
\end{abstract}
\maketitle
\section{Introduction and main results}
Associated varieties are useful tools to study the discrete decomposability of the restrictions of unitarizable simple $(\mathfrak{g},K)$-modules. Let $G$ be a reductive Lie group with Lie algebra $\mathfrak{g}$, and $G'$ a reductive subgroup with subalgebra $\mathfrak{g}'$. Take a maximal compact subgroup $K$ of $G$ such that $K':=K\cap G'$ is a maximal compact subgroup of $G'$. Denote by $\mathfrak{g}_\mathbb{C}$ and $\mathfrak{g}'_\mathbb{C}$ the complexified Lie algebras of $\mathfrak{g}$ and $\mathfrak{g}'$ respectively. The inclusion $\mathfrak{g}'_\mathbb{C}\hookrightarrow\mathfrak{g}_\mathbb{C}$ gives a projection of the dual spaces $\mathrm{pr}_{\mathfrak{g}\rightarrow\mathfrak{g}'}:\mathfrak{g}_\mathbb{C}^*\twoheadrightarrow\mathfrak{g'}_\mathbb{C}^*$.

For a unitarizable simple $(\mathfrak{g},K)$-module $X$, denote by $\mathcal{V}_{\mathfrak{g}_\mathbb{C}}(X)$ the associated variety of $X$ in the dual space $\mathfrak{g}_\mathbb{C}^*$. It is known from \cite[Theorem 3.1]{Ko4} that, if $Y$ is an simple $(\mathfrak{g}',K')$-module such that $\mathrm{Hom}_{(\mathfrak{g}',K')}(Y,X)\neq\{0\}$, then $\mathrm{pr}_{\mathfrak{g}\rightarrow\mathfrak{g}'}\mathcal{V}_{\mathfrak{g}_\mathbb{C}}(X)\subseteq\mathcal{V}_{\mathfrak{g}'_\mathbb{C}}(Y)$. However, Toshiyuki KOBAYASHI conjectured in \cite[Conjecture 5.11]{Ko7} that the inclusion is actually an equality. Namely,
\begin{conjecture}\label{17}
Let $X$ be a unitarizable simple $(\mathfrak{g},K)$-module. If $Y$ is an simple $(\mathfrak{g}',K')$-module such that $\mathrm{Hom}_{(\mathfrak{g}',K')}(Y,X)\neq\{0\}$, then $\mathrm{pr}_{\mathfrak{g}\rightarrow\mathfrak{g}'}\mathcal{V}_{\mathfrak{g}_\mathbb{C}}(X)=\mathcal{V}_{\mathfrak{g}'_\mathbb{C}}(Y)$.
\end{conjecture}
For convenience, if $\mathfrak{g}$, $\mathfrak{g}'$, and $X$ satisfy \conjectureref{17}, the author will say that \conjectureref{17} is true for the triple $(\mathfrak{g},\mathfrak{g}',X)$. \conjectureref{17} is true for $(\mathfrak{g},\mathfrak{g}',X)$ in the following cases:
\begin{enumerate}[$\bullet$]
\item $X$ is a generalized Verma module and $(\mathfrak{g},\mathfrak{g}')$ is a symmetric pair (\cite[Theorem 4.12]{Ko6});
\item $X$ is the underlying $(\mathfrak{g},K)$-module of the minimal representation of $\mathrm{O}(p,q)$ with $p+q$ even and $(\mathfrak{g},\mathfrak{g}')$ is a symmetric pair (\cite{KO});
\item $X=A_\mathfrak{q}(\lambda)$ and $(\mathfrak{g},\mathfrak{g}')$ is a symmetric pair (\cite[Theorem 8.5]{O});
\item $X$ is a highest / lowest weight simple $(\mathfrak{g},K)$-module and the natural embedding $G'/K'\hookrightarrow G/K$ is holomorphic (\cite[Theorem 7.4]{MO});
\item $X$ is the minimal holomorphic representation and $(\mathfrak{g},\mathfrak{g}')$ is a symmetric pair (\cite[Theorem 7.6]{MO}).
\end{enumerate}
One notices that most of the verifications for the \conjectureref{17} were done for symmetric pairs. In this article, the author will confirm \conjectureref{17} for Klein four symmetric pairs.
\begin{definition}\label{6}
Let $G$ (respectively, $\mathfrak{g}$) be a real simple Lie group (respectively, Lie algebra), and $\Gamma$ a Klein four subgroup of the automorphism group $\mathrm{Aut}G$ (respectively, $\mathrm{Aut}\mathfrak{g}$). Denote by $G^\Gamma$ (respectively, $\mathfrak{g}^\Gamma$) the subgroup (respectively, subalgebra) of the fixed points under the action of all elements in $\Gamma$ on $G$ (respectively, $\mathfrak{g}$). Then $(G,G^\Gamma)$ (respectively, $(\mathfrak{g},\mathfrak{g}^\Gamma)$) is called a Klein four symmetric pair. In particular, if $G$ is a simple Lie group of Hermitian type and every nonidentity element $\sigma\in\Gamma$ defines a symmetric pair of holomorphic type, then $(G,G^\Gamma)$ (respectively, $(\mathfrak{g},\mathfrak{g}^\Gamma)$) is called a Klein four symmetric pair of holomorphic type.
\end{definition}
The discrete branching laws for Klein four symmetric pairs were studied in \cite{H1}, \cite{H2}, \cite{H3}, and \cite{H4}. When $G$ is an exceptional Lie group of Hermitian type, the Klein four symmetric pairs $(G,G^\Gamma)$ of holomorphic type were classified. In this case, because of the fourth bullet item listed above, \conjectureref{17} is automatically true for $(\mathfrak{g},\mathfrak{g}^\Gamma,X)$, where $X$ is any highest / lowest weight simple $(\mathfrak{g},K)$-module. As for Klein four symmetric pairs of non-holomorphic type, $(G,G^\Gamma)=(\mathrm{E}_{6(-14)},\mathrm{Spin}(8,1))$ is a very special Klein four symmetric pair for exceptional Lie group $G$ of Hermitian type in the sense of \cite[Theorem 16]{H3}. Now the author may state the first main result of this article.
\begin{theorem}\label{19}
Let $X$ be the minimal holomorphic representation of $\mathfrak{e}_{6(-14)}$. Then \conjectureref{17} is true for the triple $(\mathfrak{e}_{6(-14)},\mathfrak{so}(8,1),X)$.
\end{theorem}
Moreover, for $(\mathfrak{g},\mathfrak{g}^\Gamma)=(\mathfrak{e}_{6(-14)},\mathfrak{so}(8,1))$, it is known from \cite[Lemma 12]{H3} that there exists an involution $\sigma\in\Gamma$ such that $(\mathfrak{g},\mathfrak{g}^\sigma)=(\mathfrak{e}_{6(-14)},\mathfrak{so}(8,2)\oplus\mathfrak{so}(2))$. Then the first direct summand $\mathfrak{so}(8,2)$ together with $\mathfrak{so}(8,1)$ forms another symmetric pair $(\mathfrak{so}(8,2),\mathfrak{so}(8,1))$. The author will confirm \conjectureref{17} for the symmetric pair $(\mathfrak{so}(8,2),\mathfrak{so}(8,1))$ with a series of unitarizable simple $(\mathfrak{so}(8,2),\mathrm{Spin}(8)\times\mathrm{Spin}(2))$-modules. Although this result was included in \cite[Theorem 4.12]{Ko6}, this provides a new way to study discrete branching laws and \conjectureref{17} for symmetric pairs by means of Klein four symmetric pairs.

In previews articles involving the discrete branching laws for Klein four symmetric pairs, the author mainly discuss the discrete decomposability of the restrictions of $(\mathfrak{g},K)$-modules. In the final part of this article, the author will discuss, for Klein four symmetric pairs $(G,G^\Gamma)$ with exceptional simple Lie groups $G$ of Hermitian type, $G^\Gamma$-admissibility of the restrictions of discrete series representations of $G$. Thus, here is the second result of this article.
\begin{theorem}\label{9}
Let $G$ be an exceptional simple Lie group of Hermitian type. If $(G,G^\Gamma)$ is a Klein four symmetric pair of non-holomorphic type, then there is no discrete series representation of $G$ which is $G^\Gamma$-admissible.
\end{theorem}
The following corollary follows from \theoremref{9} immediately.
\begin{corollary}\label{10}
Let $G$ be an exceptional simple Lie group of Hermitian type, and $(G,G^\Gamma)$ a Klein four symmetric pair. Then there exists a discrete series representation $\pi$ of $G$ which is $G^\Gamma$-admissible, if and only if $(G,G^\Gamma)$ is of holomorphic type.
\end{corollary}
\begin{proof}
The conclusion follows from \theoremref{9}, and the fact that any holomorphic / anti-holomorphic discrete series of $G$ is $G^\Gamma$-admissible.
\end{proof}
\section{Preliminary on associated varieties}
Let $\mathfrak{g}$ be a reductive Lie algebra with its complexification $\mathfrak{g}_\mathbb{C}$, and let $\{U_j(\mathfrak{g}_\mathbb{C})\}_{j\in\mathbb{Z}_{\geq0}}$ be the standard increasing filtration of the universal enveloping algebra $U(\mathfrak{g}_\mathbb{C})$. Suppose that $X$ is a finitely generated $\mathfrak{g}_\mathbb{C}$-module. A filtration $X=\bigcup\limits_{i\in\mathbb{Z}_{\geq0}}X_i$ is called a good filtration if it satisfies the following conditions:
\begin{enumerate}[$\bullet$]
\item $X_i$ is finite dimensional for any $i\in\mathbb{Z}_{\geq0}$;
\item $U_j(\mathfrak{g}_\mathbb{C})X_i\subseteq X_{i+j}$ for any $i,j\in\mathbb{Z}_{\geq0}$;
\item there exists $n\in\mathbb{Z}_{\geq0}$ such that $U_j(\mathfrak{g}_\mathbb{C})X_i=X_{i+j}$ for any $i\geq n$ and $j\in\mathbb{Z}_{\geq0}$.
\end{enumerate}
The graded algebra $\mathrm{gr}U(\mathfrak{g}_\mathbb{C}):=\bigoplus\limits_{j\in\mathbb{Z}_{\geq0}}U_j(\mathfrak{g}_\mathbb{C})/U_{j-1}(\mathfrak{g}_\mathbb{C})$ is isomorphic to the symmetric algebra $S(\mathfrak{g}_\mathbb{C})$ by the Poincar\'{e}-Birkhoff-Witt theorem and one may regard the graded module $\mathrm{gr}X:=\bigoplus\limits_{i\in\mathbb{Z}_{\geq0}}X_i/X_{i-1}$ as an $S(\mathfrak{g}_\mathbb{C})$-module. Let $\mathrm{Ann}_{S(\mathfrak{g}_\mathbb{C})}(\mathrm{gr}X):=\{f\in S(\mathfrak{g}_\mathbb{C})\mid fv=0\textrm{ for any }v\in\mathrm{gr}X\}$ and define\[\mathcal{V}_{\mathfrak{g}_\mathbb{C}}(X):=\{x\in\mathfrak{g}_\mathbb{C}^*\mid f(x)=0\textrm{ for any }f\in\mathrm{Ann}_{S(\mathfrak{g}_\mathbb{C})}(\mathrm{gr}X)\}\]which does not depend on the choice of good filtration. Then $\mathcal{V}_{\mathfrak{g}_\mathbb{C}}(X)$ is called the associated variety of $X$.

Let $\mathfrak{g}'$ be a reductive subalgebra of $\mathfrak{g}$, and then the inclusion $\mathfrak{g}'\hookrightarrow\mathfrak{g}$ gives a projection of the complexified dual spaces $\mathrm{pr}_{\mathfrak{g}\rightarrow\mathfrak{g}'}:\mathfrak{g}_\mathbb{C}^*\twoheadrightarrow\mathfrak{g'}_\mathbb{C}^*$.
\begin{proposition}\label{16}
let $X$ be an simple $(\mathfrak{g},K)$-module.
\begin{enumerate}[(1)]
\item If $Y$ is an simple $(\mathfrak{g}',K')$-module such that $\mathrm{Hom}_{(\mathfrak{g}',K')}(Y,X)\neq\{0\}$, then $\mathrm{pr}_{\mathfrak{g}\rightarrow\mathfrak{g}'}\mathcal{V}_{\mathfrak{g}_\mathbb{C}}(X)\subseteq\mathcal{V}_{\mathfrak{g}'_\mathbb{C}}(Y)$.
\item If $Y_i$ are simple $(\mathfrak{g}',K')$-modules such that $\mathrm{Hom}_{(\mathfrak{g}',K')}(Y_i,X)\neq\{0\}$ for $i=1,2$, then $\mathcal{V}_{\mathfrak{g}'_\mathbb{C}}(Y_1)=\mathcal{V}_{\mathfrak{g}'_\mathbb{C}}(Y_2)$.
\end{enumerate}
\end{proposition}
\begin{proof}
See \cite[Theorem 3.1]{Ko4} for (1) and \cite[Theorem 3.7]{Ko4} for (2).
\end{proof}
As the preparation for the next section, the author recalls the constructions of the lowest weight modules $L(\lambda)$.

Suppose that $G$ is a simple Lie group of Hermitian type, and then the Lie algebra $\mathfrak{k}$ of $K$ has a $1$-dimensional center $Z(\mathfrak{k})$. A maximal toral subalgebra $\mathfrak{t}$ of $\mathfrak{k}$ becomes a Cartan subalgebra of $\mathfrak{g}$. Moreover, there exists a characteristic element $Z\in Z(\mathfrak{k})$ such that $\mathfrak{g}_\mathbb{C}=\mathfrak{k}_\mathbb{C}+\mathfrak{p}_++\mathfrak{p}_-$ is a decomposition with respect to the eigenspaces of $Z$ on $\mathfrak{g}_\mathbb{C}$ corresponding to the eigenvalues 0, $\sqrt{-1}$, and $-\sqrt{-1}$ respectively.

Suppose that $X$ is a simple $(\mathfrak{g},K)$-module, and then set $X^{\mathfrak{p}_-}=\{v\in X\mid Yv=0\textrm{ for any }Y\in\mathfrak{p}^-\}$. Since $K$ normalizes $\mathfrak{p}_-$, $X^{\mathfrak{p}_-}$ is a $K$-submodule. Further, $X^{\mathfrak{p}_-}$ is either zero or an irreducible finite-dimensional representation of $K$. A $(\mathfrak{g},K)$-module $X$ is called a lowest weight module if $X^{\mathfrak{p}_-}\neq\{0\}$. Any highest weight simple $(\mathfrak{g},K)$-module is constructed as follows. Denote by $F(\lambda)$ the irreducible representation of $K$ with the highest weight $\lambda$. Let $\mathfrak{p}_+$ act as zero on $F(\lambda)$ and the generalized Verma module $M(\lambda)=U(\mathfrak{g}_\mathbb{C})\otimes_{U(\mathfrak{k}_\mathbb{C}+\mathfrak{p}_-)}F(\lambda)$ is a $(\mathfrak{g},K)$-module. Then the unique simple quotient $L(\lambda)$ of $M(\lambda)$ is a lowest weight simple $(\mathfrak{g},K)$-module.
\section{Proof for \theoremref{19}}
In this section, for the time being, let $G=\mathrm{E}_{6(-14)}$ with the Lie algebra $\mathfrak{g}=\mathfrak{e}_{6(-14)}$, and $\mathfrak{g}_\mathbb{C}=\mathfrak{e}_6$ the complex simple Lie algebra of type $\mathrm{E}_6$. It is known from \cite[Proposition 10]{H3} that there is a Klein four subgroup $\Gamma$ of $\mathrm{Aut}G$ such that $\mathfrak{g}^\Gamma=\mathfrak{so}(8,1)$. By \cite[Lemma 12 \& Lemma 14]{H3}, $\Gamma$ is generated by two involutive automorphisms $\sigma$ and $\tau$ with $\mathfrak{g}^\sigma\cong\mathfrak{f}_{4(-20)}$ and $\mathfrak{g}^\tau\cong\mathfrak{so}(8,2)\oplus\mathfrak{so}(2)$.

Let $K$ be a $\Gamma$-stable maximal compact subgroup of $G$, and $\mathfrak{k}$ the corresponding compact subalgebra with its complexification $\mathfrak{k}_\mathbb{C}$. Fix a $\Gamma$-stable Cartan subalgebra of the complexified Lie algebra $\mathfrak{k}_\mathbb{C}$, which is automatically a Cartan subalgebra of $\mathfrak{g}_\mathbb{C}$ because $\mathfrak{g}$ is of Hermitian type, and choose a simple root system $\{\alpha_i\mid1\leq i\leq6\}$, the Dynkin diagram of which is given in Figure 1.
\begin{figure}
\centering \scalebox{0.7}{\includegraphics{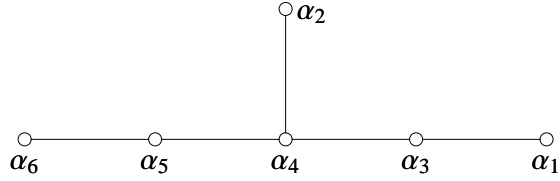}}
\caption{Dynkin diagram of $\mathrm{E}_6$.}
\end{figure}
For each simple root $\alpha_i$, denote by $\omega_i$ the fundamental weight corresponding to $\alpha_i$. Suppose that $\alpha_6$ is the noncompact simple root corresponding to the real form $\mathfrak{g}$. As described in \cite[3.11]{MO}, put $\beta_i:=\alpha_{7-i}$ for $1\leq i\leq 5$, and then $\{\beta_i\}_{i=1}^5$ form a set of simple roots for $\mathfrak{so}(10,\mathbb{C})$, the complexification of the first direct summand of $\mathfrak{g}^\tau$. Write $\mu_i$ for the fundamental weights of $\beta_i$ for $1\leq i\leq 5$.

It is known from \cite[Theorem 12.4]{EHW} that the lowest weight simple $(\mathfrak{g},K)$-module $L(3\omega_6)$ is unitarizable, which is just corresponding to the minimal holomorphic representation of $\mathfrak{g}$. Moreover, by \cite[Definition 1.3 \& Setting 2.6]{MO} that the restriction of $L(3\omega_6)$ to $\mathfrak{g}^\tau$ is discretely decomposable as a $(\mathfrak{g}^\tau,K^\tau)$-module with the decomposition\[L(3\omega_6)\cong\bigoplus\limits_{k=0}^{+\infty}L'(3\mu_1+k\mu_5)\boxtimes\mathbb{C}_{k+2}\]as $(\mathfrak{g}^\tau,K^\tau)$-modules, where $L'(3\mu_1+k\mu_5)$ denotes the lowest weight simple $(\mathfrak{so}(8,2),\mathrm{Spin}(8)\times\mathrm{Spin}(2))$-module with lowest weight $3\mu_1+k\mu_5$ and $\mathbb{C}_{k+2}$ is an $1$-dimensional module of $\mathfrak{so}(2)$. If one forgets the action of $\mathfrak{so}(2)$, then one has the discrete decomposition\[L(3\omega_6)\cong\bigoplus\limits_{k=0}^{+\infty}L'(3\mu_1+k\mu_5)\]as $(\mathfrak{so}(8,2),\mathrm{Spin}(8)\times\mathrm{Spin}(2))$-modules.
\begin{proof}[Proof for \theoremref{19}]
Since $\mathfrak{g}^\tau\cong\mathfrak{so}(8,2)\oplus\mathfrak{so}(2)$ is not compact, the center of $\mathfrak{k}$ does not centralize the whole $\mathfrak{g}^\tau$. It follows that the center of $\mathfrak{k}$ is contained in $\mathfrak{so}(8,2)$. For convenience, write $\mathfrak{h}$ for $\mathfrak{so}(8,2)$ and $\mathfrak{h}_\mathbb{C}$ for its complexification. By \cite[Theorem 7.4]{MO}, $\mathrm{pr}_{\mathfrak{g}\rightarrow\mathfrak{h}}\mathcal{V}_{\mathfrak{g}_\mathbb{C}}(L(3\omega_6))=\mathcal{V}_{\mathfrak{h}_\mathbb{C}}(L'(3\mu_1+k\mu_5))$ for any $k\in\mathbb{Z}_{\geq0}$. One the other hand, it is known from \cite[Setting 2.2]{MO} that $L'(3\mu_1)$ is the minimal holomorphic representation of $\mathfrak{h}=\mathfrak{so}(8,2)$. Moreover, by \cite[Theorem 19]{S}, $L'(3\mu_1)$ is simple as a $(\mathfrak{g}^\Gamma,K^\Gamma)$-module. By \cite[Theorem 7.6]{MO}, $\mathrm{pr}_{\mathfrak{h}\rightarrow\mathfrak{g}^\Gamma}\mathcal{V}_{\mathfrak{h}_\mathbb{C}}(L'(3\mu_1))=\mathcal{V}_{\mathfrak{g}_\mathbb{C}^\Gamma} (L'(3\mu_1))$ because $(\mathfrak{h},\mathfrak{g}^\Gamma)=(\mathfrak{so}(8,2),\mathfrak{so}(8,1))$ is a symmetric pair of anti-holomorphic type. Now one has $\mathrm{pr}_{\mathfrak{g}\rightarrow\mathfrak{g}^\Gamma}\mathcal{V}_{\mathfrak{g}_\mathbb{C}}(L(3\omega_6))=\mathrm{pr}_{\mathfrak{h}\rightarrow \mathfrak{g}^\Gamma}\circ\mathrm{pr}_{\mathfrak{g}\rightarrow\mathfrak{h}}\mathcal{V}_{\mathfrak{g}_\mathbb{C}}(L(3\omega_6))= \mathrm{pr}_{\mathfrak{h}\rightarrow\mathfrak{g}^\Gamma}\mathcal{V}_{\mathfrak{h}_\mathbb{C}}(L'(3\mu_1))=\mathcal{V}_{\mathfrak{g}_\mathbb{C}^\Gamma} (L'(3\mu_1))$. The conclusion follows from the fact that $L(3\omega_6)$ is the minimal holomorphic representation and \propositionref{16}(2).
\end{proof}
\section{A alternative way to study symmetric pairs}
Let $G$ be a noncompact reductive Lie group, $G'$ a reductive subgroup of $G$, and $G''$ a reductive subgroup of $G'$. Needless to say, let $\mathfrak{g}$, $\mathfrak{g}'$, and $\mathfrak{g}''$ be the corresponding Lie algebras with their complexifications $\mathfrak{g}_\mathbb{C}$, $\mathfrak{g}'_\mathbb{C}$, and $\mathfrak{g}''_\mathbb{C}$. Take a maximal compact subgroup $K$ of $G$ such that $K':=K\cap G'$ and $K'':=K\cap G''$ are the maximal compact subgroups of $G'$ and $G''$ respectively.
\begin{definition}\label{2}
A $(\mathfrak{g},K)$-module $X$ is said to be discretely decomposable as a $(\mathfrak{g}',K')$-module if there exists an increasing filtration $\{X_i\}_{i\in\mathbb{Z}^+}$ of $(\mathfrak{g}',K')$-modules such that $\bigcup\limits_{i\in\mathbb{Z}^+}X_i=X$ and $X_i$ is of finite length as a $(\mathfrak{g}',K')$-module for any $i\in\mathbb{Z}^+$.
\end{definition}
\begin{proposition}\label{3}
Let $X$ be a simple $(\mathfrak{g},K)$-module. Then $X$ is discretely decomposable as a $(\mathfrak{g}',K')$-module if and only if there exists a simple $(\mathfrak{g}',K')$-module $Y$ such that $\mathrm{Hom}_{(\mathfrak{g}',K')}(Y,X)\neq\{0\}$. Moreover, suppose that $X$ is a unitarizable simple $(\mathfrak{g},K)$-module. Then $X$ is discretely decomposable as a $(\mathfrak{g}',K')$-module if and only if it is isomorphic to a direct sum of simple $(\mathfrak{g}',K')$-modules.
\end{proposition}
\begin{proof}
See \cite[Lemma 1.3 \& Lemma 1.5]{Ko4}.
\end{proof}
\begin{lemma}\label{18}
Let $X$ be a unitarizable simple $(\mathfrak{g},K)$-module, which is discretely decomposable as a $(\mathfrak{g}',K')$-module and is also discretely decomposable as a $(\mathfrak{g}'',K'')$-module. Suppose that $Y$ is a simple $(\mathfrak{g}',K')$-module such that $\mathrm{Hom}_{(\mathfrak{g}',K')}(Y,X)\neq\{0\}$.
\begin{enumerate}[(1)]
\item Then $Y$ is discretely decomposable as a $(\mathfrak{g}'',K'')$-module.
\item If \conjectureref{17} is true for $(\mathfrak{g},\mathfrak{g}'',X)$, then \conjectureref{17} is also true for $(\mathfrak{g}',\mathfrak{g}'',Y)$.
\end{enumerate}
\end{lemma}
\begin{proof}
Since the unitarizable simple $(\mathfrak{g},K)$-module $X$ is discretely decomposable as a $(\mathfrak{g}',K')$-module and $Y$ is a simple $(\mathfrak{g}',K')$-module such that $\mathrm{Hom}_{(\mathfrak{g}',K')}(Y,X)\neq\{0\}$, $X\cong Y\oplus Y'$ as $(\mathfrak{g}',K')$-modules for some unitarizable $(\mathfrak{g}',K')$-module $Y'$ by \propositionref{3}. Then the natural projection shows that $\{0\}\neq\mathrm{Hom}_{(\mathfrak{g}',K')}(X,Y)\subseteq\mathrm{Hom}_{(\mathfrak{g}'',K'')}(X,Y)$. On the other hand, since $X$ is also discretely decomposable as a $(\mathfrak{g}'',K'')$-module, $X\cong\bigoplus\limits_im(i)Z_i$ is a direct sum of simple $(\mathfrak{g}'',K'')$-modules with multiplicities $m(i)$ by \propositionref{3}. Now $\{0\}\neq\mathrm{Hom}_{(\mathfrak{g}'',K'')}(X,Y)\cong\mathrm{Hom}_{(\mathfrak{g}'',K'')} (\bigoplus\limits_im(i)Z_i,Y)\subseteq\prod\limits_im(i)\mathrm{Hom}_{(\mathfrak{g}'',K'')} (Z_i,Y)$. Hence, there must be some $Z_i$ with $m(i)>0$ such that $\mathrm{Hom}_{(\mathfrak{g}'',K'')}(Z_i,Y)\neq\{0\}$, and $Y$ is discretely decomposable as a $(\mathfrak{g}'',K'')$-module by \propositionref{3}. This proves (1).

If \conjectureref{17} is true for $(\mathfrak{g},\mathfrak{g}'',X)$, then $\mathrm{pr}_{\mathfrak{g}\rightarrow\mathfrak{g}''}\mathcal{V}_{\mathfrak{g}_\mathbb{C}}(X)=\mathcal{V}_{\mathfrak{g}''_\mathbb{C}}(Z)$ for any simple $(\mathfrak{g}'',K'')$-module $Z$ with $\mathrm{Hom}_{(\mathfrak{g}'',K'')}(Z,X)\neq\{0\}$. If $Y$ is simple $(\mathfrak{g}',K')$-module with $\mathrm{Hom}_{(\mathfrak{g}',K')}(Y,X)\neq\{0\}$, then $Y$ is discretely decomposable as a $(\mathfrak{g}'',K'')$-module by (1), and $\mathrm{pr}_{\mathfrak{g}'\rightarrow\mathfrak{g}''}\mathcal{V}_{\mathfrak{g}'_\mathbb{C}}(Y)\subseteq\mathcal{V}_{\mathfrak{g}''_\mathbb{C}}(Z)$ by \propositionref{16}. Thus one has $\mathcal{V}_{\mathfrak{g}''_\mathbb{C}}(Z)=\mathrm{pr}_{\mathfrak{g}\rightarrow\mathfrak{g}''}\mathcal{V}_{\mathfrak{g}_\mathbb{C}}(X)= \mathrm{pr}_{\mathfrak{g}'\rightarrow\mathfrak{g}''}\circ\mathrm{pr}_{\mathfrak{g}\rightarrow\mathfrak{g}'}\mathcal{V}_{\mathfrak{g}_\mathbb{C}}(X)\subseteq \mathrm{pr}_{\mathfrak{g}'\rightarrow\mathfrak{g}''}\mathcal{V}_{\mathfrak{g}'_\mathbb{C}}(Y)\subseteq\mathcal{V}_{\mathfrak{g}''_\mathbb{C}}(Z)$, and thus $\mathrm{pr}_{\mathfrak{g}'\rightarrow\mathfrak{g}''}\mathcal{V}_{\mathfrak{g}'_\mathbb{C}}(Y)=\mathcal{V}_{\mathfrak{g}''_\mathbb{C}}(Z)$. Since any simple $(\mathfrak{g}'',K'')$-submodule of $Y$ is also a simple $(\mathfrak{g}'',K'')$-submodule of $X$, (2) follows from \propositionref{16}(2).
\end{proof}
Retain the notations as in the last section, one obtains the following result.
\begin{corollary}\label{20}
The highest weight simple $(\mathfrak{so}(8,2),\mathrm{Spin}(8)\times\mathrm{Spin}(2))$ modules $L'(3\mu_1+k\mu_5)$ for $k\in\mathbb{Z}_{\geq0}$ are all discretely decomposable as $(\mathfrak{so}(8,1),\mathrm{Spin}(8))$-modules, and \conjectureref{17} is true for $(\mathfrak{so}(8,2),\mathfrak{so}(8,1),L'(3\mu_1+k\mu_5))$ for $k\in\mathbb{Z}_{\geq0}$.
\end{corollary}
\begin{proof}
Since $L(3\omega_6)$ is discretely decomposable as a $(\mathfrak{so}(8,2),\mathrm{Spin}(8)\times\mathrm{Spin}(2))$-module with direct summands $L'(3\mu_1+k\mu_5)$ for $k\in\mathbb{Z}_{\geq0}$ and $L(3\omega_6)$ is also discretely decomposable as a $(\mathfrak{so}(8,1),\mathrm{Spin}(8))$-module, the first statement follows from \lemmaref{18}(1) immediately. The second statement follows from \theoremref{19} and \lemmaref{18}(2).
\end{proof}
\begin{remark}\label{11}
Notice that $(\mathfrak{so}(8,2),\mathfrak{so}(8,1))$ is a symmetric pair and each lowest weight simple module $L'(3\mu_1+k\mu_5)$ is actually a simple generalized Verma module. The result of \corollaryref{20} is contained in \cite[Theorem 4.1 \& Theorem 4.12]{Ko6}. However, \lemmaref{18} offers an alternative way to study the discrete branching laws and \conjectureref{17} for symmetric pairs through Klein four symmetric pairs, and \corollaryref{20} can be regarded as an example.
\end{remark}
\section{Proof for \theoremref{9}}
In the final part of this article, the author will discuss the restrictions of discrete series representations for Klein four symmetric pairs $(G,G^\Gamma)$ with exceptional simple Lie groups $G$ of Hermitian type.

Let $G$ be a noncompact reductive Lie group, and $G'$ a reductive subgroup of $G$.
\begin{definition}\label{1}
Let $\pi$ be an irreducible unitary representation of $G$ on a Hilbert space. Then $\pi$ is said to be $G'$-admissible if $\pi$ decomposes as a Hilbert direct sum of irreducible unitary representations of $G'$ with finite multiplicities, i.e.,\[\pi\cong\widehat{\bigoplus\limits_{\tau\in\widehat{G'}}}m(\tau)\tau\]with the multiplicities $m(\tau)\in\mathbb{Z}_{\geq0}$, where $\widehat{G'}$ denotes the unitary dual of $G'$.
\end{definition}
\begin{proposition}\label{4}
Let $\pi$ be a discrete series representation of $G$. Then the following conditions are equivalent:
\begin{enumerate}[(1)]
\item $\pi$ is $K'$-admissible;
\item $\pi$ is $G'$-admissible;
\item the underlying $(\mathfrak{g},K)$-module $\pi_K$ is discretely decomposable as a $(\mathfrak{g}',K')$-module with finite multiplicities.
\end{enumerate}
\begin{proof}
The directions $(1)\Rightarrow(3)$ and $(3)\Rightarrow(2)$ follow from \cite[Proposition 1.6]{Ko4} and \cite[Theorem 2.7]{Ko1} respectively, which hold for general unitary representations $\pi$. If $\pi$ is a discrete series representation of $G$, then the direction $(2)\Rightarrow(1)$ follows from \cite[Corollary 2.5]{DGV}.
\end{proof}
\end{proposition}
\begin{remark}\label{5}
The direction $(3)\Rightarrow(1)$ was also proved in \cite[Theorem 1.3]{ZL}.
\end{remark}
\begin{lemma}\label{8}
Let $(G,G^\Gamma)$ be a Klein four symmetric pair. Suppose that $\pi$ is a discrete series representation of $G$. If $\pi$ is $G^\Gamma$-admissible, then its underlying $(\mathfrak{g},K)$-module $\pi_K$ is discretely decomposable as a $(\mathfrak{g}^\sigma,K^\sigma)$-module for any $\sigma\in\Gamma$.
\end{lemma}
\begin{proof}
This follows from \cite[Theorem 1.2]{Ko2} and \propositionref{4} immediately.
\end{proof}
\begin{proof}[Proof for \theoremref{9}]
Since $(G,G^\Gamma)$ is supposed to be of non-holomorphic type, there exists a nonidentity element $\sigma\in\Gamma$ such that $(G,G^\sigma)$ is a symmetric pair of anti-holomorphic type. Now assume that there exists a discrete series representation $\pi$ of $G$ which is $G^\Gamma$-admissible, and then by \lemmaref{8}, its underlying $(\mathfrak{g},K)$-module $\pi_K$ is discretely decomposable as a $(\mathfrak{g}^\sigma,K^\sigma)$-module. According to the classification result \cite[Theorem 5.2 \& Table 1]{KO2}, the only possible symmetric pair of anti-holomorphic type is $(\mathfrak{g},\mathfrak{g}^\sigma)=(\mathfrak{e}_{6(-14)},\mathfrak{f}_{4(-20)})$. It is well known that as the underlying $(\mathfrak{g},K)$-module of a discrete series representation, $\pi_K=A_\mathfrak{b}(\lambda)$ for some $\theta$-stable Borel subalgebra of $\mathfrak{g}$. However, according to the classification result in \cite[Table C.3]{KO1}, $(\mathfrak{g},\mathfrak{g}^\sigma)=(\mathfrak{e}_{6(-14)},\mathfrak{f}_{4(-20)})$ is not of discrete series type; in other words, there does not exist any $A_\mathfrak{b}(\lambda)$ for a $\theta$-stable Borel subalgebra of $\mathfrak{g}$ which is discretely decomposable as a $(\mathfrak{g}^\sigma,K^\sigma)$-module. Thus, one obtains a contradiction.
\end{proof}

\end{document}